\documentclass{amsart}
\usepackage{amsmath}
\usepackage[all]{xy}
\newtheorem{thm}{Theorem}

\newtheorem{cor}[thm]{Corollary}
\newtheorem{prop}[thm]{Proposition}
\newtheorem{rem}[thm]{Remark}
\newtheorem{lem}[thm]{Lemma}
\def\grdeg{\deg^{\operatorname{gr}}}
\newtheorem{thmABC}{Theorem}
\newtheorem{ques}[thmABC]{Question}

\newcommand\trdeg{{\operatorname{trdeg}}}

\newcommand\ideal[1]{{\langle{#1}\rangle}}
\newcommand\sg[1]{{\ideal{#1}}}

\newcommand\lcm{{\operatorname{lcm}}}
\newcommand\Tref[1]{{Theorem~\ref{#1}}}

\newcommand\Pref[1]{{Proposition~\ref{#1}}}

\newcommand\ra{{\rightarrow}}

\newcommand\eq[1]{{(\ref{#1})}}

\newcommand\M[1][n]{{\operatorname{M}_{#1}}}
\newcommand\diag[1]{{\operatorname{diag}(#1)}}
\newcommand\Z{{\mathbb{Z}}}

\long\def\forget#1\forgotten{{}}
\title[Makar-Limanov's problem]{Makar-Limanov's problem on values of polynomials on matrices}

\author{Louis H. Rowen}
\author{Uzi Vishne}
\address{Department of Mathematics, Bar-Ilan University, Ramat-Gan 52900, Israel}


\begin{document}

\thanks{$^*$The   authors thank Lenny Makar-Limanov for raising this question, pointing out the connection to the Freiheitsatz, and other helpful conversations.}

\thanks{$^*$ The authors thank Be'eri Greenfeld for helpful comments}
\thanks{$^*$The authors were supported by the ISF grant 1994/20. The first author was supported by the Anshel Pfeffer Chair}

\begin{abstract}
Suppose $F$ is an infinite field and let $f \in F\sg{X_1, \dots,X_m}$ be a noncommutative polynomial. Partially answering a query of Makar-Limanov, we show that there  are numbers $d$ and $m'$ such that, if $F$ is closed under taking $d$th roots, for any $n \ge  m'$ there are matrices $A_1,\dots,A_m$ in~$M_n(F)$ such that $f(A_1,\dots,A_m)$ is upper triangular with $n-m'$ prescribed diagonal entries.

When f is homogeneous, $f(A_1,\dots,A_m)$ is diagonal with $n-m'$ prescribed diagonal entries.

When f is multilinear, we can take $d=1$ and $m' = [\frac{m-1}{2}]$,  and the upper left  $(n-m')\times (n-m')$ piece of $f(A_1,\dots,A_m)$ can be taken to be $\diag{\beta_1,\dots, \beta_{n-m'}}$, for indeterminates $\beta_i$.

Furthermore, if $f$ is not a polynomial identity of $ k \times k $ matrices, then at least $ n - k $ characteristic values of $ f(A_1,\dots,A_m) $ may be taken to be algebraically independent.
\end{abstract}

\maketitle

\section{Introduction}
Let  $F\sg{X_1, \dots,X_m}$ be the free associative algebra in  noncommuting indeterminates over an arbitrary field $F$. Let $f(\M[n](F))$ denote the set of all evaluations $f(A_1,\dots,A_m)$ for $A_1,\dots,A_m \in \M[n](F)$.
Makar-Limanov raised the question:

\begin{ques}\label{originalQ}
Let $f \in F\sg{X_1, \dots,X_m}$ be a polynomial, where $F$ is algebraically closed. Is the minimal rank of evaluations in  $f(\M[n](F))$ bounded when $n \ra \infty$?
\end{ques}

 \begin{rem}
      As Makar-Limanov pointed out, a positive answer would provide a characteristic-free Freiheitssatz for associative algebras, that each nonconstant element $f$ of the free associative algebra generates a proper ideal. (This is   known so far only in characteristic 0, with a more complicated argument using algebraically closed skew fields, in \cite{ML}.)

      Indeed, were the ideal $\ideal{f}$ not proper, there would be polynomials $ p_i, q_i $, $1 \leq i \leq \ell$ (for suitable $\ell$) such that $1 = \sum_{i} p_i f q_i$. But if   Question~\ref{originalQ} has a positive answer, then there would be evaluations in $\M[n](F)$  for which the identity matrix
$$I_{n} = \sum_{i=1}^{\ell} p_i(A_1, \dots ,A_m) f(A_1, \dots, A_m) q_i(A_1, \dots, A_m)$$
 has rank $\leq C\ell$ for a suitable constant $C$, a contradiction when $n$ is large.
 \end{rem}

In this paper, we address the following stronger question:

\begin{ques}\label{T3}
Let $f \in F\sg{X_1, \dots,X_m}$. Is there a number $m'$ such that, for any~$n$ there is an evaluation $A \in f(\M[n](F))$ such that $A$ is diagonal with prescribed entries except for possibly the final $m'$ entries?
\end{ques}

We give a partial answer, proving that for suitable $d$ (only depending on the degrees of monomials in $f$), if $F$ is closed under taking  roots of polynomials of degree $d$, then $f(\M[n](F))$ has upper triangular evaluations whose diagonal entries (and thus whose characteristic values) are $\beta_i$ for $1\le i \le n-m$.

Also we answer  Question~\ref{T3} when  $f$ is completely homogeneous.
When  $f$ is multilinear, we can take $d=1$ and $m' = [\frac{m-1}{2}],$
and then the $\beta_i$ can be taken to be indeterminates over $F,$  and the upper left  $(n-m')\times (n-m')$ piece of $f(A_1,\dots,A_m)$ can be taken to be $\diag{\beta_1,\dots, \beta_{n-m'}}$.

\def\ra{{\rightarrow}}
\def\Q{{\mathbb{Q}}}

\subsection{Statements of results}$ $

We say that $f$ is {\bf{very rich}} if $f(\M[n](F))$ contains   diagonal matrices with an arbitrary diagonal. Similarly, for a fixed number $m'$, we say that $f$ is~{\bf{very $m'$-rich}} if for any $n$, $f(\M[n](F))$ contains   diagonal matrices with an arbitrary diagonal, after some adjustment of the last $m'$ entries along the diagonal. Conjugating by a permutation matrix, this condition is equivalent to stating that every diagonal matrix appears in $f(\M[n](F))$ after adjusting  $m'$ entries in any fixed position.

\begin{thm}\label{T3.multi}
Over any field $F$, every multilinear polynomial $f(X_1,\dots,X_m)$ is very $(m-1)$-rich.
\end{thm}

\begin{thm}\label{T3a}
Let $f(X_1,\dots,X_m)$ be a completely homogeneous polynomial, with $d_i = \deg_{X_i}f$. For some $i_0$, take $m' = \lcm(d_{i_0},\sum d_j -  d_{i_0})$ and suitable $d \le d_{i_0}$ (as specified in the proof). If $F$ is closed under taking $d$ roots, then~$f$ is very $m'$-rich.
\end{thm}

Settling for a weaker condition, we say that $f$ is {\bf{rich}} if $f(\M[n](F))$ contains triangular matrices with an arbitrary diagonal; and {\bf{$m'$-rich}} if for any $n$, $f(\M[n](F))$ contains
triangular matrices with an arbitrary diagonal, after some adjustment of the last
$m'$ entries along the diagonal.

\begin{thm}\label{T3b}
For $d$ and $m'$ to be described in the proof, if $F$ is closed under taking roots of polynomials of degree~$d$,  then any~$f$ is $m'$-rich.
\end{thm}

The proofs, which are completely elementary, are given in \S\ref{se1}. They involve one particular kind of substitution of matrix units, which we call \textbf{standard}.

\begin{cor}
Notation as in Question~\ref{T3}, The upper  $(n-m')\times (n-m')$ part of  $f$ evaluated on $M_n(F)$ takes on the value of any  $(n-m')\times (n-m')$ matrix $A$ having distinct eigenvalues ${\beta_1}, \dots,{\beta_{n-m'}} $.

When f is completely homogeneous of degree $d_i$ in the $i$-th indeterminate, we can fix any $i_0$ and take $d = d_{i_0}$ and $m' = \operatorname{lcm}(d_{i_0},\sum d_j - d_{i_0})$.

\end{cor}
\begin{proof}
  The  matrix $A$ is diagonalizable to $\diag{\beta_1}, \dots,\diag{\beta_{n-m'}} $. On the other hand $f$ has a triangular evaluation $f(x_1,\dots, x_m)$ with eigenvalues ${\beta_1}, \dots,{\beta_{n-m'}} $, so we can conjugate the upper $(n-m')\times (n-m')$ portion of $f(x_1,\dots, x_m)$ to $A$.

\end{proof}

We also obtain  a better bound   but for a  weaker result:

\begin{prop}\label{PropA}
Suppose $ f(X_1,\dots, X_m) \in F\sg{X_1,\dots, X_m}$ is a noncommutative polynomial, 
 which is  noncentral on $k \times k$ matrices  over an algebraically closed field $F$. Then the set of characteristic values $ \alpha_1,\dots,\alpha_n$ of $f(Y_1,\dots, Y_m)$, the evaluation of $f$ on generic matrices, has transcendence degree $\ge n - k$ (over~$F$), for each $n \ge k$.
\end{prop}
\section{Proofs of the theorems}\label{se1}

We  take commuting indeterminates $\mu_{i,j}$ ($i=1,\dots,m$, $j = 1,\dots,n)$, and work over the field $\hat F = F(\mu_{i,j})$, where we shall specialize the $\mu_{i,j}$ to values in~$F$ as described below.
Recall that the matrix algebra $\M[n](F)$ is spanned by matrix units $e_{i,j},1 \le i,j\le n,$ and is $\Z$-graded by assigning $\grdeg(e_{i,j}) = i-j$.
The homogeneous components are $M_\delta = \sum_j F e_{j,j+\delta}$, and in particular $M_0$ is the space of diagonal matrices.
We make the additional restriction that the $X_1,\dots,X_m$ will be assigned to be gr-homogeneous matrices $x_1,\dots,x_m$ under this grading. We call this a \textbf{standard substitution}.

\begin{proof}[Proof of \Tref{T3.multi}]
We handle this multilinear case separately in order to get intuition, and also since the optimal result is rather clearcut. 

Suppose we fix $q_1,\dots,q_m \in \mathbb \Z$ with $\sum _{i=1}^{m} q_i  = 0$, and take the homogeneous substitution
\begin{equation}\label{setx}
x_i = \sum_j \mu_{i,j}e_{j,j+q_i}
\end{equation}
(summing over all $j$ for which $e_{j,j+q_i}$ is defined, i.e., $1\leq j,j+q_i\leq n$).
Then $f(x_1,\dots,x_m)$ is homogeneous, with $$\grdeg(f(x_1,\dots,x_m)) = \sum \grdeg(x_i) = \sum q_i  = 0,$$
so $f(x_1,\dots,x_m)$ is diagonal.
But what are the  entries of the evaluation?

Fixing a permutation $\pi \in S_m$ and an index $j \in [1,n]$, define $(j_1,\dots,j_m)$ by $j_1 = j$,\ $j_2 = j_1+q_{\pi 1}$,\ $j_3 = j_2+q_{\pi 2}$,\ ...,\ $j_m = j_{m-1}+q_{\pi m}$. Since $\sum q_i = 0$, the cycle closes with $j_m = j_1$.

We say that $\pi$ is {\bf{kosher}} at $j$ if all the entries in this vector are in the range~$[1,n]$, so that the matrix units $e_{j_i,j_{i+1}}$ are defined.  Let $C_j$ denote the set of $\pi \in S_n$ that are kosher at $j$.

Write $f(X_1,\dots,X_m) = \sum_{\pi \in S_m} \alpha_{\pi} X_{\pi 1}\cdots X_{\pi m}$, where $\alpha_{\pi} \in F$ and $\alpha_{(1)} = 1$. Along the diagonal, at the $j,j$ entry, we get
\begin{equation*}\begin{aligned}
f&(x_1,\dots,x_m)_{j,j}  =  \left(\sum_\pi \alpha_{\pi}  x_{\pi 1}\cdots x_{\pi m}\right) e_{j,j} \\
& =   \sum_{\pi} \alpha_{\pi}
\sum_{j_1} \mu_{\pi 1, j_1}e_{j_1,j_1+q_{\pi 1}}
\sum_{j_2} \mu_{\pi 1, j_2}e_{j_2,j_2+q_{\pi 1}}
\cdots
\sum_{j_m} \mu_{\pi 1, j_m}e_{j_m,j_m+q_{\pi 1}}
e_{j,j}.
\end{aligned}
\end{equation*}
In this sum, a nonzero product is obtained precisely when $\pi$ is kosher at $j$.
Therefore,
\begin{equation}\label{thefjj}
f_{j,j} = f(x_1,\dots,x_m)_{j,j} = \sum_{\pi \in C_j} \alpha_{\pi} \mu_{\pi 1, j_1}
\mu_{\pi 2, j_2}  \mu_{\pi 3, j_3} \cdots \mu_{\pi m, j_m}.
\end{equation}



We take
$(q_1,\dots,q_m) = (1,1,1,\dots,-m')$, where $m' = m-1$, so the matrices $x_1,\dots,x_m$ are well-defined by \eq{setx}. Note the exceptional role of $x_m$, and hence we denote the $\mu_{m,j}$ by $\nu_j$.


It is easy to see that every term in the evaluation \eq{thefjj} comes from a circuit of the matrix units, based in the vertex $j$, with each path corresponding to a distinct~$x_i$, $1\le i \le m$, where $m-1$ edges increase the index by $1$, and one edge completes the circuit by dropping the index by $m-1$. Indeed, there can only be one drop, which matches the other edges ascending by~$1$. In other words, for each $f_{j,j}$ (for $j\leq n+1-m$), 
we only have contributions which come from the unique circuit from~$j$ to~$j$ including the drop precisely once, in which $\nu_j$ occurs (in the final position). 
Other contributions would involve monomials in~$\nu_j$ and in the $\mu_{i,j'}$   
for $j'<j.$ By definition, different circuits correspond to different monomials.


Since $x_m$ has degree 1 in $f$, the coefficient of $\nu_j$ is a nonzero polynomial in the $\mu_{i,j}$ regardless of the $\alpha_{\pi}$, because the monomials are distinct (since they all correspond to different circuits) and $\alpha_1 = 1$. Thus for any $j_0$ we can solve $\mu_{mj_0}$ for any desired value $f(x_1,\dots,x_m)_{j_0,j_0}$, in rational expressions in the $\mu_{i,j}$ and the $\nu_j$ for $j<j_0,$
 and $\nu_{1},\dots,\nu_{j_0-1}$ can be evaluated inductively so that the $f(x_1,\dots,x_m)_{j,j}$, $j=1,\dots,n+1-m$, hit any desired targets.

To see this more explicitly, consider the transformation from the vector $(\nu_1,\dots,\nu_n)$ to $(f_{11},\dots,f_{nn})$, which is linear over the variables $\mu_{i,j}$ (for $i<m$). For $j = n$, the only kosher permutations start with a negative step $-m'$, so the only $\nu_j$ participating in the expression for $f_{n,n}$ is $\nu_n$. For $j=n-1$ a kosher permutation may start by either a negative step of $-m'$ or by a positive step of $+1$ followed by $-m'$, and consequently $f_{n-1,n-1}$ is a linear combination of $\nu_{nn}$ and $\nu_{n-1,n-1}$. In this manner we see that the transformation from $(\nu_j)$ to $(f_{jj})$ is represented by a triangular matrix (in terms of the $\mu_{ij}$), which can be solved for the last $n-m'$ variables.

This argument degenerates for  $m=1$, but then the assertion is obvious.
\begin{figure}
$$\xymatrix@C=34pt{
 1  \ar@/^2ex/@{->}[r]|{\mu_{11}}
 \ar@/^4ex/@{->}[r]|{\mu_{21}}  \ar@/^6ex/@{->}[r]|{\mu_{31}}
& 2 \ar@/^2ex/@{->}[r]|{\mu_{12}} \ar@/^4ex/@{->}[r]|{\mu_{22}}  \ar@/^6ex/@{->}[r]|{\mu_{32}}
& 3 \ar@/^2ex/@{->}[r]|{\mu_{13}} \ar@/^4ex/@{->}[r]|{\mu_{23}}  \ar@/^6ex/@{->}[r]|{\mu_{33}}
& 4 \ar@/^2ex/@{->}[r]|{\mu_{14}} \ar@/^4ex/@{->}[r]|{\mu_{24}}  \ar@/^6ex/@{->}[r]|{\mu_{34}} \ar@/^3ex/@{->}[lll]|{\mu_{41}}
& 5 \ar@/^2ex/@{->}[r]|{\mu_{15}} \ar@/^4ex/@{->}[r]|{\mu_{25}}  \ar@/^6ex/@{->}[r]|{\mu_{35}} \ar@/^3ex/@{->}[lll]|{\mu_{42}}
& 6 \ar@/^2ex/@{->}[r]|{\mu_{16}} \ar@/^4ex/@{->}[r]|{\mu_{26}}  \ar@/^6ex/@{->}[r]|{\mu_{36}} \ar@/^3ex/@{->}[lll]|{\mu_{43}}
& 7 \ar@/^2ex/@{->}[r]|{\mu_{17}} \ar@/^4ex/@{->}[r]|{\mu_{27}}   \ar@/^6ex/@{->}[r]|{\mu_{37}} \ar@/^3ex/@{->}[lll]|{\mu_{44}}
& 8 \ar@/^3ex/@{->}[lll]|{\mu_{45}} \ar@/^2ex/@{->}[r]|{\mu_{18}} \ar@/^4ex/@{->}[r]|{\mu_{28}}   \ar@/^6ex/@{->}[r]|{\mu_{38}} \
& 9 \ar@/^3ex/@{->}[lll]|{\mu_{46}}
}
$$
\caption{The matrices $x_1,\dots,x_n$ for $n=9$ and $m=4$}
\end{figure}
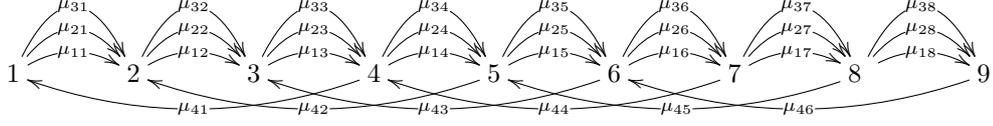
\end{proof}

\begin{rem}\label{prec} Before proceeding to the totally homogeneous case, let us review what was needed in the substitution in the multilinear case. The condition guaranteeing that $f(x_1,\dots,x_m)$   is diagonal is that the sum of the graded degrees of the $x_i$ is zero.

We also do not want our circuit to contain subcycles, so that it has no nontrivial automorphisms, in order that there be no cancellation in the coefficients of $\nu_j$, and choose this most efficiently, for $m' =m-1$ by \begin{equation}
    \label{m} x_i =  \sum \mu_{i,j} e_{j,j+q_i},
\end{equation}
where $q_i = 1$  and $q_{m'} = -m',$ and ``kosher'' implies only the matrix units $e_{i,j}: 1\le i,j \le n$ are used.
We get a circuit since $\sum q_i d_i = 0.$ Since we have graded the indices,  the increases in the indices of the matrix unit substitutions of the $x_i$ for $i<m$ match the decrease in the matrix units of $x_m$.

 An algebraic viewpoint is that we need to solve  a system of equations   whose coefficient matrix is in triangular form, and the product of the diagonal terms (nonzero polynomials in the $\mu_{i,j})$ is nonzero, and thus we can solve this system.
\end{rem}

\begin{proof}[Proof of \Tref{T3a}]
Let $f$ be a completely homogeneous polynomial. For ease of notation, we take $i_0 = m.$ We take the same approach as in the multilinear case,  Again we take the $x_i$ as in \eqref{m}, of graded degree $q_i,$ now with $q_i>0$ for $1\le i < m$, and  with $\sum_{i=1}^{m-1} q_id_i =  -q_md_m$ to guarantee that we   have diagonal values. We rearrange the substitution as
\begin{equation}
    \label{w1} \begin{aligned}
        x_i &= \sum _{k\ge 1} \mu_{i+(k-1)q_i , i + kq_i }e_{i+(k-1)q_i , i + kq_i }, \\ x_m &=\sum_{k\ge 0} \mu_{j+(k+1)d_m ,j+kd_m}e_{j+(k+1)d_m ,j+kd_m},
    \end{aligned}
\end{equation}
summed where the matrix units are defined. Now we   take $\nu_j$ to be the highest indexed $\mu$ in the sum for $x_m$, i.e., $\nu_j =   \mu_{j+q_md_m , j+q_m(d_m-1)}.$ (But this could occur up to $d_m$ times, such as for $f = (X_1X_2)^d.$) Again to solve for   $f_{j,j}$ we have a triangular linear  system in the $\nu_j$ in which the coefficients of the $\nu_j$ are monic polynomials with coefficients in the other $\mu_{i,j},$  so this system has a solution, if we are allowed to take roots in $F$ of polynomials of degree  $d$, where $d$ is the number of highest vertices of the circuit $\Gamma$.
\end{proof}

\begin{proof}[Proof of \Tref{T3b}]
Since the polynomial no longer is homogeneous, we cannot get a diagonal matrix by a standard substitution.

But we are able to get an upper triangular matrix, which is what we claim.
Consider homogeneous matrix unit substitutions $x_i,$ $1\le i <m.$ For any monomial $h=h (X_1,\dots, X_m, \dots, X_{m})$ of $f,$ let $d_{i,h}$ be the degree of $X_i$ in $h$,  $1\le i \le m,$ and let $d_h = \sum _{i=1}^{m-1} d_{i,h}.$ First assume that $d_h \ne 0$ for each  monomial~$h$.

We define
\begin{equation}
    \label{mon} \phi(h) = \frac{d_h}{d_{m,h}},
\end{equation}
(which we put as $\infty$ if $d_{m,h} =  0),$ and $\phi(f) = \max_{\text{monomials } h \text{ of }  f}\phi(h).$

Let $\bar f$ be the sum of those monomials $h$  of $f$  for which $\phi(h)= \phi(f).$ (Note that $\bar f$ need not be completely homogeneous since for example we could have $f = X_1X_3 + X_2^2 X_3^2$, in which case $\phi(h)=1$ for each monomial of $f$, and $\bar f = f$.) Nevertheless, for the standard substitution of~\Tref{T3a},
$f(x_1, \dots, x_m)$ is upper triangular, and its evaluation on the diagonal is $\bar f(x_1, \dots, x_m).$

But the eigenvalues of an upper triangular matrix are precisely its diagonal elements. Thus, to produce a desired set of eigenvalues, we can discard all the monomials $h$
with $\phi(h)< \phi(f).$
In this case,  we only get occurrences of $\nu_j$, the   highest indexed $\mu$, in the sum for $x_m$, for monomials $h$ of maximal degree. (For the other monomials we only get occurrences $\mu_{j'}$ for $j'<j.)$ Thus, we get a triangular set of equations, with the same leading coefficients in the $\nu_j,$ so we have essentially reduced to Theorem~\ref{T3a}.

Thus we may assume that there is a monomial where $X_1,\dots,X_{m-1}$ do not appear; then we can specialize $X_1,\dots,X_{m-1}$ to $0$ and have a nonzero polynomial $f(X_m)= \alpha_i X_m^i.$  In this case we even can conclude that $f$ is rich! Namely, take a diagonal matrix $\diag{\beta_i}$ and solve $f(\gamma_i) = \beta_i$ for $1\le i\le n,$ so take $x = (\gamma_i).$
\end{proof}

 \begin{rem}$ $
    \begin{itemize}
    \item     Different choices of $q_i$, and perhaps mixing positive, zero, and negative choices for the $q_i$ could provide a lower (i.e., more efficient) value of $m'.$

        \item We see in the substitution \eqref{w1}, the evaluation of the diagonal of $f$ is ``almost'' controlled by the evaluation of a homogeneous component.

        \item Our emphasis on Question~\ref{T3} is for $n\gg m$, but one naturally is interested in the lower $m' \times m'$ part of the evaluation. This is related to the L'vov -Kaplansky problem about all the possible evaluations of a polynomial on matrices. Ironically the case $m>n$ seems to be more difficult to handle, cf.~ \cite[Theorems 4,5]{KBR3} and \cite{KBR4}.
    \end{itemize}
 \end{rem}

\section{Proof of \Pref{PropA}}

This requires a fact about transcendence.

\begin{lem}\label{ind} Suppose $A =F[c_1,\dots, c_n]$ is an affine domain and $\pi: A\to B$ is a projection such that, writing $\bar a = \pi (a),$ $\bar c_n =0$ and $\bar c_1, \dots, \bar c_u$ are algebraically independent. Then $c_1, \dots,   c_u, c_n$ are algebraically independent.
 \begin{proof}
     Any equation
$$0 = \sum_{i} g_i(c_1, \dots, c_u) c_n^i$$
of minimal degree specializes to
$$0 = \sum_{i}  {g}_i(\bar{c}_1,\dots, \bar{c}_u) \bar{c}_n^i =  {g}_0(\bar{c}_1,\dots,\bar{c}_u),$$
implying $g_0 = 0 $ since $\bar{c}_1,\dots,\bar{c}_u$ are algebraically independent; then we can cancel~$c_n,$ contrary to minimality of $\deg g $. (We do not rule out the possibility $ u= 0 $, in which case the $ g_i $ would be constant.)
 \end{proof}
\end{lem}
\begin{proof}({\it of Proposition~\ref{PropA}.})
We show that the set of characteristic values $ \alpha_1, \dots, \alpha_n $ of~$ f(X_1, \dots, X_m)$, where $X_\ell = (\xi_{ij}^{\ell})$ are the generic matrices, has transcendence degree $\ge n-k$ (over $F$), for every $n \ge k$. Let $c_1 =\sum \alpha_i, \ \dots,\  c_n =\prod \alpha_i $ be the values of the elementary symmetric functions evaluated on $\alpha_1, \dots, \alpha_n$. Then each $ c_i \in F(\alpha_1,\dots,\alpha_n)$, and furthermore $c_1,\dots,c_n$ are the coefficients of the characteristic polynomial of $f(X_1,\dots,X_m)$. (In particular $c_n\ne 0$ since $ F \langle X_1, \dots X_m \rangle $  is a domain.) Thus each of $\alpha_1,\dots,\alpha_n$ is algebraic over $F(c_1,\dots,c_n)$, and thus letting $t = \trdeg F(c_1,\dots, c_n )$, we also have $t = \trdeg F(\alpha_1,\dots,\alpha_n)$.

It remains to show that $t \ge n-k$. We prove this by induction. For $n = k$ the assertion is trivial. For $n > k$ consider the specialization
$$\xi_{in}^{(\ell)} \mapsto 0, \ \xi_{nj} ^{(\ell)}\mapsto 0 ,\ \xi_{nn}^{(\ell)} \mapsto 0 \quad (i,j=1,\dots,n-1, \ \ell = 1,\dots,m).$$

Writing $\overline{\phantom{x}}$ for the image under this specialization, we can identify $\overline{X_1}, \dots, \overline{X_m}$ as generic $(n-1) \times (n-1)$ matrices, 
so $f(\overline{X_1}, \dots, \overline{X_m})$ is singular, implying $\bar{c}_n = 0$, and $\bar{c}_1,\dots,\bar{c}_{n-1}$ are the characteristic coefficients of $f(\overline{X_1}, \dots, \overline{X_m})$. By induction
$$
\trdeg \,F(\bar{c}_1,\dots,\bar{c}_{n-1}) \geq  n-1-k.$$ 

Take a transcendence base $\bar{c}_1,\dots,\bar{c}_u$ of $F(\bar{c}_1,\dots,\bar{c}_{n-1})$. Then $c_n$ is algebraically independent of $c_1, \dots, c_u$ over $F$, by Lemma~\ref{ind}. Thus
$$t \geq 1 + u \ge n-k,$$
as desired.
\end{proof}

\begin{ques}
What is the lower bound for the transcendence degree  $t$? We suspect that $t \geq n - 1$ for all $n > k$.
\end{ques}

\begin{rem}$ $
    \begin{itemize}
      \item Our original effort was to verify Question~\ref{T3} as a corollary of \Pref{PropA}, by specializing the eigenvalues, but Makar-Limanov has pointed out that even if elements are algebraically independent, one may not have freedom in their specialization, e.g., $\beta_1(\beta_1-\beta_2)$ and $\beta_2(\beta_1-\beta_2)$ cannot be specialized to the same nonzero element.

  \item \cite{BV}  has a slightly weaker version of   \Pref{PropA} for  nonconstant noncommutative rational functions.
    \end{itemize}
\end{rem}

\end{document}